\documentclass[11pt,a4paper]{article}
\usepackage[utf8]{inputenc}
\usepackage[english]{babel}
\usepackage{amsmath}
\usepackage{amsfonts}
\usepackage{amssymb}
\usepackage{amsthm}
\usepackage{graphicx}
\usepackage{epsfig}
\usepackage{cite}
\usepackage[pagewise]{lineno}
\usepackage{dsfont}
\usepackage[all]{xy}
\usepackage{indentfirst}
\usepackage{url}

\newtheorem{theorem}{Theorem}[section]

\newtheorem{Corolario}[theorem]{Corollary}

\newtheorem{teo}{Theorem}
\newtheorem{Theorem B}{Theorem B}
\newtheorem{remark}[theorem]{Remark}

\numberwithin{equation}{section}

\newcommand{\be}{\begin{eqnarray}}
\newcommand{\en}{\end{eqnarray}}
\newcommand{\no}{\nonumber}

\newcommand{\R}{\mathbb{R}}

\newcommand{\dive}{\text{div}}

\newcommand{\Addresses}{{
\bigskip
\footnotesize
Instituto Federal Goiano, Campus Trindade,  Brazil \\
\textit{E-mail address:} \texttt{adriano.bezerra@ifgoiano.edu.br}
\vspace{.3cm}\\
Universidade de S\~ao Paulo, S\~ao Carlos, Brazil \\
\textit{E-mail address:} \texttt{manfio@icmc.usp.br} \\
}}

\begin{document}

\author{A. C. Bezerra and  F. Manfio
\footnote{Fernando Manfio is supported by Fapesp, Grant 2016/23746-6.}}

\title{Umbilicity of constant mean curvature hypersurfaces into
space forms}

\date{\today}
\maketitle

\vspace{2ex}


\begin{abstract}
In this paper we establish conditions on the length of the traceless part
of the second fundamental form of a complete constant mean curvature
hypersurface immersed in a space of constant sectional curvature in
order to show that it is totally umbilical.
\end{abstract}

\section{Introduction} \label{sec:Introd}

The celebrated result by Bernstein \cite{BE} asserts that the only complete
minimal graphs in $\mathbb{R}^{3}$ are planes. Bernstein's theorem remains
valid for complete minimal graphs in $\mathbb{R}^{n+1}$ provided that
$n\leq7$, as state the works of Fleming \cite{F}, De Giorgi \cite{D}, Almgren
\cite{A} and Simons \cite{S}. However, the restriction on the dimension is
necessary, as shown by a counterexample due Bombieri, De Giorgi and
Giusti \cite{BGG}. The stability of the entire minimal graphs leads us to the
natural question of whether a complete stable minimal hypersurface in
$\mathbb{R}^{n+1}$, with $n\leq7$, is a hyperplane. It was proved
independently by do Carmo and Peng \cite{CP}, Fischer-Colbrie and
Schoen \cite{FCS} that a complete stable minimal surface in
$\mathbb{R}^{3}$ must be a plane. A generalization of this theorem for
higher dimensions was obtained by do Carmo and Peng:

\begin{teo}[\cite{CP}, \cite{FCS}] \label{Theorem A1}
Let $M^{n}$ be a minimal hypersurface in $\mathbb{R}^{n+1}$ . Assume that
$M^{n}$ is stable, complete and that
\be\no
\lim_{R\rightarrow +\infty} \frac{\int_{B_p(R)}|A|^{2}}{R^{2q+2}} = 0, \quad
q<\sqrt{\frac{2}{n}}.
\en
Then $M^{n}$ is a hyperplane.
\end{teo}

Here, $B_{p}(R)$ denotes the geodesic ball of radius $R$ centered at
$p \in M^{n}$ and $A$ is the second fundamental form of $M^{n}$.
Some partial answers in order to generalize Theorem \ref{Theorem A1}
can be found in \cite{INS}, \cite{INS2}. In particular, the authors discuss
in \cite{INS2} the concept of entropy associated to the volume of
geodesic balls in a complete noncompact Riemannian manifold.

Theorem \ref{Theorem A1} has been extended to hypersurfaces with
constant mean curvature by Alencar and do Carmo \cite{AC}. A crucial
point is to replace the second fundamental form $A$ of the immersion
by the traceless second fundamental form $\phi=-A+HI$, here $H$
denotes the mean curvature of $M^{n}$. More precisely,

\begin{teo}[\cite{AC}] \label{Theorem A2}
Let $M^{n}$ be a complete noncompact hypersurface in $\mathbb{R}^{n+1}$,
$n\leq5$, with constant mean curvature $H$. Assume that M is strongly stable
and that
\be\no
\lim_{R\rightarrow +\infty} \frac{\int_{B_p(R)}|\phi|^{2}}{R^{2q+2}} = 0,
\quad q<\frac{1}{6n+1}.
\en
Then $M^{n}$ is a hyperplane.
\end{teo}

Theorem \ref{Theorem A2} were slightly improved by do Carmo
and Zhou \cite{CZ}, by changing the integral condition in Theorem
\ref{Theorem A2} to a slightly weaker condition, and improving the
restriction on the dimension to $n\leq6$. There are several other
interesting results that generalize Theorem \ref{Theorem A1},
including for other ambient spaces; for instance, see \cite{BM},
\cite{BW}, \cite{INS}, \cite{INS2}, \cite{NQ}, \cite{NQX}, among
others.

The goal of this paper is to give an improved version of Theorem
\ref{Theorem A2} in order to obtain umbilical hypersurfaces in spaces
of constant sectional curvature. We exchange the stability condition
by a condition in the norm of the traceless second fundamental $\phi$,
that makes no restrictions on the dimension of the hypersurface.

%spaces without restriction in dimension, since Alencar-do Carmo and
%Carmo-Zhou required $n\leq5$ and $n\leq6$, respectively, and we
%consider a more general environment than $\mathbb{R}^{n}$.

Given a complete hypersurface $M^{n}$ with constant mean curvature
$H$ in a Riemannian manifold $\mathbb{M}^{n+1}(c)$, with constant
sectional curvature $c$, we will denote the first eigenvalue
$\lambda_{1}(L_{\triangle+|A|^{2}+\overline{Ric}(\nu)},M^{n})$ of the
strong stability operator of $M^{n}$ by $\overline{\lambda}_{1}$ (see
Section \ref{sec:Preliminaries} for details). In our first main result
we will consider the polynomial
\be\label{1.1}
P_{H}(x)= x^{2}+\theta Hx-\beta,
\en
where the constants  $\theta$ and $\beta$ are defined as
\be\label{1.2}
\theta:= \dfrac{n(n-2)}{\sqrt{n(n-1)}} \quad\text{and}\quad
\beta:=n(H^{2}+c).
\en
The polynomial $P_{H}$, which depends on $n$, $H$ and $c$, will help us
to understand the relationship between the norm of the traceless second
fundamental $\phi$ and the positive root of $P_{H}$, giving us conditions
for $M^n$, in the case of constant mean curvature, to be umbilical.
More precisely,

\begin{theorem} \label{th1}
Let $M^{n}$ be a complete hypersurface with constant mean curvature $H$,
immersed in a Riemannian manifold $\mathbb{M}^{n+1}(c)$ of constant
sectional curvature $c$, with $H^{2}+c>0$ if $c<0$. Assume that
\be
\label{1.3} \lim_{R\rightarrow +\infty} \frac{\int_{B_p(R)}
|\phi|^{2q+2}}{R^2} = 0, \ \ q>0,
\en
and
\be\label{1.4}
\sup_{x \in M} |\phi|<r_{H},
\en
where $r_{H}$ is a positive root of the polynomial $P_{H}$ defined in
\eqref{1.1}. Then $M$  is totally umbilical.
\end{theorem}

Theorem \ref{th1} generalizes \cite[Theorem 1.2]{BM} for the case of
umbilical hypersurfaces with constant mean curvature. In fact, in
\cite{BM}, the authors obtain a rigidity result for totally geodesic minimal
hypersurfaces in the hyperbolic space $\mathbb{H}^{n}$, with similar
conditions to \eqref{1.3} and \eqref{1.4}.

\vspace{.2cm}

In our second main result we obtain a condition that relates the norm of
the traceless second fundamental form $\phi$ with the first eigenvalue
$\overline{\lambda}_{1}$ of the strongly stability operator, in order to
obtain the umbilicity of the hypersurface. Furthermore, we replace the
condition stability for $M^{n}$ in Theorem \ref{Theorem A2} for a weaker
condition (see \eqref{1.7}) in the first eigenvalue of the strongly stability
operator. For that, we need to define the polynomial
\be\label{1.5}
P_{H,\overline{\lambda}_{1}}(x):=-x^{2}-\theta Hx+
\left(1+\frac{1}{q}\right)\beta+\dfrac{1}{q}\overline{\lambda}_{1},
\en
where $q$ is a positive constant.

\begin{theorem} \label{th2}
Let $M^{n}$ be a complete hypersurface with constant mean curvature $H$
immersed in a Riemannian manifold $\mathbb{M}^{n+1}(c)$ of constant
sectional curvature  $c$. Suppose that
\be\label{1.6}
\lim_{R\rightarrow +\infty} \frac{\int_{B_p(R)}|\phi|^{2}}{R^{2q+2}} = 0,
\quad\text{with}\quad \frac{-q^{2}+q+1}{q}<\frac{2}{n}.
\en
If the first eigenvalue of the strongly stability operator satisfies
\be\label{1.7}
\overline{\lambda}_{1}>-n(q+1)(H^{2}+c)
\en
and
\be\label{1.8}
\sup_{x \in M} |\phi|<r_{H},
\en
where $r_{H}$  is a positive root of $P_{H,\overline{\lambda}_{1}}$ defined in (\ref{1.5}), then $M^{n}$ is totally umbilical.
\end{theorem}

\begin{remark}
In Theorem $\ref{th2}$, when $n\leq5$ and $M^{n}$ is a complete
noncompact hypersurface immersed in a form space $\mathbb{Q}^{n+1}_c$,
condition $H^{2}+c>0$ together with equation \eqref{1.7} imply that the
first eigenvalue $\overline{\lambda}_{1}$ must satisfy $\overline{\lambda}_{1}<0$,
that is, $M^n$ must be unstable. In fact, in this case, it follows from
\cite[Corollary 6.3]{INS} that there is no complete noncompact stable
hypersurface with constant mean curvature $H\neq0$ in $\mathbb{R}^{n+1}$ or $\mathbb{S}^{n+1}$, or in $\mathbb{H}^{n+1}$, when
 $H^{2}>g(n)$, where
%= \mathbb{R}^{n+1}$ or $\mathbb{S}^{n+1}$, or $H^{2}>g(n)$,  $\mathbb{Q}^{n+1}(c)=\mathbb{H}^{n+1}$ and
\be\no
g(n)=\frac{\left(\frac{n+2}{n}\right)^{2}-1}{\left(\frac{n+2}{n}\right)^{2}-\frac{n^{2}}{4(n-1)}}.
\en
\end{remark}

%The Remark 1.3 toghether with

%Some partial answers for generalization of the Do Carmo-Peng Theorem (Theorem A) can also be found in \cite{INS2}, where the authors make a study of the concepts of entropy associated to the volume of geodesic balls
%in  a complete, noncompact Riemannian manifold.

In the case that $M^n$ is a minimal hypersurface of a Riemannian manifold
$\mathbb{M}^{n+1}(c)$ with constant sectional curvature $c$, the integral
condition \eqref{1.6} becomes the same as in Theorem \ref{Theorem A1}.
Thus, when $\mathbb{M}^{n+1}(c)=\mathbb{R}^{n+1}$, condition \eqref{1.7}
means that $M$ is stable, and we conclude that $M^{n}$ is a hiperplane.
In this case, we recover the do Carmo-Peng's Theorem.

\begin{Corolario} \label{c4}
Under the conditions of Theorem \ref{th2}, if $M^{n}$ is a minimal hypersurface
in the Euclidean space $\mathbb{R}^{n+1}$, then $M^{n}$ is a hiperplane.
\end{Corolario}

A natural question is whether condition \eqref{1.8} in Theorem \ref{th2}
can be removed, since it does not appear in Theorems \ref{Theorem A1}
and \ref{Theorem A2}. In fact, this is the case for low dimensions.

\begin{theorem} \label{th5}
Let $M^{n}$ be a complete hypersurface, $n=2,3$, with constant
mean curvature $H$, immersed in a Riemannian manifold
$\mathbb{M}^{n+1}(c)$ of constant sectional curvature $c$. Suppose that
\be\label{1.9}
\lim_{R\rightarrow +\infty} \frac{\int_{B_p(R)}|\phi|^{2q+2}}{R^2} = 0,
\quad q<\frac{1}{3}\left(\sqrt{\frac{2(6-n)}{n}}-1\right).
\en
If
\be\label{1.10}
\overline{\lambda}_{1}>n\left[\dfrac{(q+1)^{2}}{2(n+2nq+2)}(1-2n(H^{2}+c)
-(H^{2}+c)\right],
\en
then $M^{n}$ is totally umbilical.
\end{theorem}

The paper is organized as follows. In the next section we recall some basic
analytic tools of Riemannian manifolds. In particular, we recall the first
eigenvalue of the Laplacian operator and the notion of stability to 
hypersurfaces with constant mean curvature. In Section 
\ref{sec:inequalities} we establish some important inequalities that will be
used throughout the paper. Finally, in the last section we prove the
results established in Section \ref{sec:Introd}.

\section{Preliminaries} \label{sec:Preliminaries}

Given a complete Riemannian manifold $M^n$ and a smooth function
\linebreak
$\beta:M\to\R$, we denote by $\Delta$ and $\lambda_1(L_\beta,M)$
the Laplacian operator acting on the space $C^\infty(M)$ and the first
eigenvalue of the operator $L_\beta=\Delta+\beta$, respectively.
More precisely, $\lambda_1(L_\beta,M)$ is defined by
\be	\label{2.1}
\lambda_{1}(L_{\beta},M)= \inf_{f\in C_0^{\infty}(M), f\neq 0}
\frac{\int_{M}(|\nabla f|^{2}-\beta f^{2})}{\int_{M} f^{2}}.
\en
Note that, when $\beta=0$, $\lambda_1(L_0,M)$ recover the usual
first eigenvalue of $M$.

When $M^n$ is an oriented complete hypersurface of a Riemannian
manifold $\overline{M}^{n+1}$, with a unit normal vector field $\nu$,
let $A$ be the shape operator of $M^n$ with respect to $\nu$ given by
\be\no
AX = \overline\nabla_X\nu,
\en
for all tangent vector field $X$, where $\overline{\nabla}$ denotes the
Levi-Civita connection of $\overline{M}^{n+1}$. Fixed a point $p\in M$,
the principal curvatures $k_1\leq k_2\leq\ldots\leq k_n$ of $M$ at $p$
with respect to $\nu(p)$ are defined as the eigenvalues of $A(p)$.
The mean curvature $H$ of $M^n$ at $p$ is defined by
\[
H:=\frac{1}{n}\sum_{i=1}^{n}k_{i},
\]
and the square of the second fundamental form of $M^n$ is given by
\[
|A|^{2}:=\sum_{i=1}^{n}k_{i}^{2}.
\]
Moreover, we will denote by $\overline{Ric}(\nu)$ the Ricci curvature of
$\overline{M}^{n+1}$ along $\nu$.
%Here $\overline{Ric}(\nu)=\sum_{i=1}^{n}K(e_{i}\wedge \nu)$.

Let $L$ be the second order differential operator on $M$ given by
\be\no
L=\triangle+|A|^{2}+\overline{Ric}(\nu).
\en
Associated to the operator $L$, we define a quadratic form on the functions
$f\in C^\infty(M)$ that have support on a compact domain $K\subset M$ by
\be\no
I(f)=-\int_{M} fLf.
\en
For each such compact domain $K$, we define the index $ind_{L}K$ of $L$
on $K$ as the maximal dimension of a subspace where
$I$ is negative definite. The index $indM$ of $L$ in $M$ is the number defined by
\be\no
indM = \sup_{K\subset M} \ ind_{L}K,
\en
where the supremum is taken over all compact domains $K\subset M$. $M$ is said to be stable if its index is null, that is, $indM =0$.

When $M$ is minimal, equivalently we say that $M$ is {\em stable} if for every piecewise
smooth functions $f:M\to\mathbb{R}$ with compact support, we have
\be\label{2.2}
\int_{M}|\nabla f|^{2} -\int_{M}(|A|^{2}+\overline{Ric}(\nu))f^{2}\geq0,
\en
where $|\nabla f|$ denotes the gradient of $f$ in the induced metric.

\vspace{.2cm}

The notion of stability has been extended to hypersurfaces with constant
mean curvature. More precisely, we say that a constant mean curvature
hypersurface $M$ is {\em strongly stable} if \eqref{2.2} holds for all
piecewise smooth functions $f:M\to\mathbb{R}$ with compact
support; $M$ is said to be {\em weakly stable} if \eqref{2.2} holds for all
piecewise smooth functions $f:M\to\mathbb{R}$ with compact support
and $\int_{M}f=0$.

\section{Some inequalities} \label{sec:inequalities}

In this section we will establish some inequalities that will be used throughout
the paper. In what follows, we assume that $M^{n}$ is an oriented complete
hypersurface with constant mean curvature $H$, immersed in a Riemannian
manifold $\mathbb{M}^{n+1}(c)$ of constant sectional curvature $c$. In order
to study such hypersufaces, is more convenient to modify the second
fundamental form and to introduce a new linear operator
$\phi:T_{p}M\to T_{p}M$ given by
\be\no
\langle\phi X,Y\rangle=-\langle A X,Y\rangle+H\langle X,Y\rangle.
\en
Note that $\phi$ can also be diagonalized as $\phi e_{i} = \mu_{i}e_{i}$
and $tr\phi=0$. Moreover,
\be\no
|\phi|^{2}:=\sum_{i=1}^{n}\mu_{i}^{2}=\dfrac{1}{2n}\sum_{i,j=1}^{n}(k_{i}-k_{j})^{2}.
\en
Thus $|\phi|^{2}$ measures how far $M$ is from being totally umbilical.

Associated to the operator $\phi$, Ilias, Nelli and Soret obtained in 
\cite[Corollary 3.1]{INS} the following inequality:
\be \label{2.3}
|\phi|\Delta |\phi|+|\phi|^{4} + \frac{n(n-2)}{\sqrt{n(n-1)}}H|\phi|^{3}-
n(H^{2}+c)|\phi|^{2}\geq\dfrac{2}{n}|\nabla|\phi||^{2}.
\en
With the constants $\theta$ and $\beta$ defined in \eqref{1.2}, the inequality
\eqref{2.3} becomes 
\be\label{2.4}
|\phi|\Delta |\phi|\geq\dfrac{2}{n}|\nabla|\phi||^{2}  - |\phi|^{4}-\theta H|\phi|^{3}+
\beta|\phi|^{2}.
\en
Let $q$ be a positive constant. Since that $\Delta |\phi|^{q}=\dive(\nabla |\phi|^{q})$,
we have
\be \label{2.5}
|\phi|^{q}\triangle|\phi|^{q}=
\left(1-\frac{1}{q}\right)|\nabla|\phi|^{q}|^{2}+q|\phi|^{2(q-1)}|\phi|\triangle|\phi|.
\en
Now, replacing \eqref{2.4} in \eqref{2.5}, we get
\be \label{2.6}
\left(1+\frac{2-n}{nq}\right)|\nabla|\phi|^{q}|^{2}\leq q\left(|\phi|^{2}
+\theta H|\phi|-\beta\right)|\phi|^{2q} +|\phi|^{q}\triangle|\phi|^{q}.
\en
Since that $f\in C_{0}^{\infty}(M)$, multiplying the inequality \eqref{2.6} by
$f^{2}|\phi|^{2}$ and integrating over $M$, we get
\begin{equation} \label{2.7}
\begin{aligned}
\left(1+\frac{2-n}{nq}\right)\int_{M}|\nabla|\phi|^{q}|^{2}f^{2}|\phi|^{2}
\leq & q \int_{M}\left(   |\phi|^{2}+\theta H|\phi|-\beta\right)f^{2}|\phi|^{2q+2}\\
+& \int_{M}f^{2}|\phi|^{q+2}\triangle|\phi|^{q}.
\end{aligned}
\end{equation}
By divergence theorem, we have
\be
\no0&=&\int_{M}div (f^{2}|\phi|^{q+2}\nabla|\phi|^{q}) 
=\int_{M}2f|\phi|^{q+2}\langle\nabla f,\nabla|\phi|^{q}\rangle \\
\no&+&\int_{M}(q+2) f^{2}|\phi|^{q+1}\langle\nabla |\phi|,\nabla|\phi|^{q}\rangle
+\int_{M}f^{2}|\phi|^{q+2}\triangle|\phi|^{q}.
\en
Thus, the inequality \eqref{2.7} becomes
\be
&&\no\left(1+\frac{2-n}{nq}\right)\int_{M}|\nabla|\phi|^{q}|^{2}f^{2}|\phi|^{2}
 \leq  q\int_{M}\left(   |\phi|^{2}+\theta H|\phi|-\beta\right)f^{2}|\phi|^{2q+2}\\
 & &\no-2\int_{M}f|\phi|^{q+2}\langle\nabla f,\nabla|\phi|^{q}\rangle
 -\dfrac{(q+2)}{q}\int_{M} f^{2}|\phi|^{2}|\nabla|\phi|^{q}|^{2}.
\en
By Cauchy-Schwarz inequality, and Young inequality with $\epsilon$, we 
can find a constant $\epsilon>0$ such that
\begin{eqnarray} \label{2.8}
\begin{aligned}
\left(\frac{n(2q+1)-2}{nq}-2\epsilon\right)&\int_{M}|\nabla|\phi|^{q}|^{2}f^{2}|\phi|^{2}
\leq +\frac{1}{2\epsilon}\int_{M}  |\phi|^{2q+2}|\nabla f|^{2} \\
& + q\int_{M}\left( |\phi|^{2}+\theta H|\phi|-\beta\right)f^{2}|\phi|^{2q+2}.
\end{aligned}
\end{eqnarray}

On the other hand, from the definition of the stability operator $L_{|A|^{2}+nc}$ 
defined over $M$, we have the following characterization for the first eigenvalue
$\overline{\lambda}_{1}:=\overline{\lambda}_{1}(L_{|A|^{2}+nc},M^{n})$:
\[
\int_{M}|A|^{2}f^{2}+nc\int_{M}f^{2}+\overline{\lambda}_{1}\int_{M}f^{2}
\leq  \int_{M} |\nabla f|^{2}.
\]
Since $|A|^{2}=|\phi|^{2}+nH^{2}$, we get
\begin{equation} \label{eq:normaA}
\int_{M}|\phi|^{2}f^{2}+\beta\int_{M}f^{2}+\overline{\lambda}_{1}\int_{M}f^{2}
\leq  \int_{M} |\nabla f|^{2}.
\end{equation}
Replacing $f\in C_{0}^{\infty}(M)$ by $f=f|\phi|^{q+1}$ in \eqref{eq:normaA},
we have
\begin{eqnarray} \label{2.9}
\begin{aligned}
\int_{M}|\phi|^{2q+4}f^{2}+&(\beta+\overline{\lambda}_{1})\int_{M}|\phi|^{2q+2}f^{2}
\leq\frac{(q+1)^{2}}{q^{2}} \int_{M}f^{2}|\nabla |\phi|^{q}|^{2} |\phi|^{2} \\
+& 2(q+1) \int_{M} f|\phi|^{2q+1}\langle\nabla f,\nabla\phi\rangle+  \int_{M}|\phi|^{2q+2}|\nabla f|^{2}.
\end{aligned}
\end{eqnarray}
Again, making use of the Cauchy-Schwarz inequality, and Young inequality
with $\epsilon$, we get
\[
2\int_{M} f|\phi|^{2q+1}\langle\nabla f,\nabla\phi\rangle\leq 
2\left(\varepsilon \int_{M}f^{2}|\phi|^{2q}|\nabla|\phi||^{2}+
\dfrac{1}{4\varepsilon }\int_{M}|\phi|^{2q+2}|\nabla f|^{2}\right).
\]
Therefore, we obtain the following inequality:
\begin{eqnarray} \label{2.10}
\begin{aligned}
&\int_{M}|\phi|^{2q+4}f^{2}+(\beta+\overline{\lambda}_{1})\int_{M}|\phi|^{2q+2}f^{2} \\ 
&\leq\left(\left(1+\dfrac{1}{q}\right)^{2}+\dfrac{\varepsilon(q+1)}{q^{2}} \right)
\int_{M}f^{2}|\nabla |\phi|^{q}|^{2}|\phi|^{2} \\
&+\left(1+\dfrac{q+1}{\varepsilon}\right)\int_{M}|\phi|^{2q+2}|\nabla f|^{2}.
\end{aligned}
\end{eqnarray}

\section{Proof of Theorems}

\begin{proof}[Proof of Theorem \ref{th1}.]  
By equation \eqref{2.8}, we have
\begin{eqnarray*}
&&\left(\frac{n(2q+1)-2}{nq}-2\epsilon\right)
\int_{M}|\nabla|\phi|^{q}|^{2}f^{2}|\phi|^{2}\\
&& \leq  q\int_{M}\left( |\phi|^{2}+\theta H|\phi|-
\beta\right)f^{2}|\phi|^{2q+2} +\dfrac{1}{2\epsilon}\int_{M}
|\phi|^{2q+2}|\nabla f|^{2}.
\end{eqnarray*}
Since the polynomial $P_{H}(x)=x^{2}+\theta Hx-\beta$ has two roots
$r_{0}<0<r_{1}$, and if $\sup_{x\in M} |\phi|<r_{1}$, we have
\[
|\phi|^{2}+\theta H|\phi|-\beta<0.
\]
Thus, we get
\begin{equation} \label{eq:proofTheo1}
\left(\frac{n(2q+1)-2}{nq}-2\epsilon\right)\int_{M}|\nabla|\phi|^{q}|^{2}f^{2}
|\phi|^{2}\leq\frac{1}{2\epsilon}\int_{M}  |\phi|^{2q+2}|\nabla f|^{2}.
\end{equation}
Let $f$ be a nonnegative smooth function defined on the interval $[0,+\infty)$
such that $f\equiv1$ in $[0,R]$, $f\equiv0$ in $[2R,+\infty)$, and 
$|f'|\leq\frac{2}{R}$. Consider the composition $f\circ r$, where $r$ is the 
distance function from the point $p$. It follows from \eqref{eq:proofTheo1}
that we can choose a constant $\epsilon>0$ such that
\be \label{3.1}
\int_{B_{p}(R)}|\nabla|\phi|^{q}|^{2}|\phi|^{2} \leq
\frac{C}{R^{2}}\int_{B_{p}(2R)} |\phi|^{2q+2},
\en
where $C$ is a positive constant that depends only on $n$, $q$ and 
$\epsilon$. Taking the limit $R\rightarrow+\infty$, it follows from \eqref{1.3}
that
\be\label{3.2}
&&\int_{M}|\nabla|\phi|^{q}|^{2}|\phi|^{2}=0,
\en
and $|\phi|$ equal to a constant $k$ along $M$. Suppose $k>0$. Since
$\sup |\phi|< r_{1}$ and using inequality \eqref{2.8}, we have
\[
0\leq  q\int_{M}\left( |\phi|^{2}+\theta H|\phi|-\beta\right)f^{2}|\phi|^{2q+2}
\leq\frac{1}{2\epsilon}\int_{M}  |\phi|^{2q+2}|\nabla f|^{2}.
\]
Taking again the function $f\circ r$, defined above, and applying the
hypothesis \eqref{1.3}, it follows that
\[
\int_{M}\left( |\phi|^{2}+\theta H|\phi|-\beta\right)|\phi|^{2q+2}=0.
\]
This shows that $|\phi|^{2}+\theta H|\phi|-\beta=0$, which is a contradiction.
Therefore, we have $|\phi|\equiv0$ and the conclusion follows.
\end{proof}

\vspace{.1cm}

\begin{proof}[Proof of Theorem \ref{th2}.]
From condition \eqref{1.6}, we have
\[
\frac{-q^{2}+q+1}{q}<\dfrac{2}{n}.
\]
Thus, we can find a constant $\epsilon>0$ such that
\[
\left(1+\frac{1}{q}\right)^{2}+\frac{\epsilon(q+1)}{q^{2}} <1+
\frac{2-n+n(q+2)}{nq}-2\epsilon.
\]
In this case, joining the inequalities \eqref{2.8} and \eqref{2.10}, we obtain
the following inequality:
\begin{eqnarray} \label{3.3}
\begin{aligned}
\frac{1}{q}\int_{M}|\phi|^{2q+4}f^{2} + &
\int_{M}\left( -|\phi|^{2}-\theta H|\phi|+\left(\dfrac{q+1}{q}\right)\beta +
\frac{1}{q}\overline{\lambda}_{1}\right)f^{2}|\phi|^{2q+2} \\
\leq & \frac{1}{q}\left(1+\frac{2(q+1)+1}{2\epsilon}\right)
\int_{M}|\phi|^{2q+2}|\nabla f|^{2}.
\end{aligned}
\end{eqnarray}
Considering the polynomial $P_{H,\overline{\lambda}_{1}}$ defined in
\eqref{1.5}, we have:
\begin{eqnarray*}
\begin{aligned}
P_{H,\overline{\lambda}_{1}}(x)=&-x^{2}-\theta Hx+\left(\dfrac{q+1}{q}\right)\beta+\dfrac{1}{q}\lambda_{1}\\
=&-x^{2}-\theta Hx+\left(\dfrac{q+1}{q}\right)n(H^{2}+c)+\dfrac{1}{q}\lambda_{1}.
\end{aligned}
\end{eqnarray*}
Thus, in terms of $P_{H,\overline{\lambda}_{1}}$, we can rewrite \eqref{3.3}
as
\be \label{3.4}
\frac{1}{q}\int_{M}|\phi|^{2q+4}f^{2} + 
\int_{M}P_{H,\overline{\lambda}_{1}}(|\phi|)f^{2}|\phi|^{2q+2}
\leq C\int_{M}|\phi|^{2q+2}|\nabla f|^{2},
\en
with $C=C(q,\varepsilon)$. Since \eqref{1.7} is satisfied, we can see that
the polynomial $P_{H,\overline{\lambda}_{1}}(|\phi|)$ has two roots 
$r_{0}<0<r_{1}$. If $\sup |\phi|< r_{1}$, then $P_{H,\overline{\lambda}_{1}}> 0$,
and we get the following inequality:
\be\label{3.5}
\int_{M}|\phi|^{2q+4}f^{2}\leq \overline{C}\int_{M}|\phi|^{2q+2}|\nabla f|^{2}, \ 
\forall f\in C_{0}^{\infty}(M).
\en
We can replace $f$ by $f^{q+1}$ in \eqref{3.5} in order to get
\be\no 
\int_{M}|\phi|^{2(q+2)}f^{2(q+1)} &\leq&
\overline{ C}\int_{M}|\phi|^{2(q+1)}|\nabla f^{q+1}|^{2}\\
&=&\no\overline{ C}(q+1)^{2}\int_{M}|\phi|^{\frac{2q(q+2)}{q+1}}
|\phi|^{\frac{2}{q+1}}f^{2q}|\nabla f|^{2}.
\en
Let us denote $C_{1}=\overline{C}(q+1)^{2}$. It follows from
H\"{o}lder inequality that
\be
\no\int_{M}|\phi|^{2(q+2)}f^{2(q+1)} &\leq &C_{1}\left[\int_{M}\left(|\phi|^{\frac{2q(q+2)}{q+1}}f^{2q}\right)^{\frac{q+1}{q}}\right]^{\frac{q}{q+1}}
\left[\int_{M}\left(|\phi|^{\frac{2}{q+1}}|\nabla f|^{2}\right)^{q+1}\right]^{\frac{1}{q+1}}\\
\no&\leq &C_{1}\left[\int_{M}|\phi|^{2(q+2)}f^{2(q+1)}\right]^{\frac{q}{q+1}}
\left[\int_{M}|\phi|^{2}|\nabla f|^{2(q+1)}\right]^{\frac{1}{q+1}},
\en
that is
\be
\no\left(\int_{M}|\phi|^{2(q+2)}f^{2(q+1)}\right)^{\frac{1}{q+1}}&\leq &C_{1}\left(\int_{M}|\phi|^{2}|\nabla f|^{2(q+1)}\right)^{\frac{1}{q+1}}.
\en
Proceeding as in the proof of Theorem \ref{th1}, we have
\[
\int_{B_p(R)}|\phi|^{2(q+2)}\leq \frac{C_{1}}{R^{2q+2}}\int_{B_p(2R)}|\phi|^{2}.
\]
Finally, from \eqref{1.6}, we conclude that $|\phi|\equiv 0$, and the conclusion
of Theorem \ref{1.2} follows.
\end{proof}

\vspace{.1cm}

\begin{proof}[Proof of Corollary \ref{c4}.] 
Since $M^{n}$ is a minimal hypersurface in the Euclidean space
$\mathbb{R}^{n+1}$, we have
\[
\beta=n(H^{2}+c)=0 \quad\text{and}\quad |\phi|^{2}=|A|^{2}.
\]
Thus, inequality \eqref{3.3} implies that
\begin{eqnarray*}
\begin{aligned}
\int_{M}|A|^{2q+4}f^{2} \leq & 
\int_{M}|A|^{2q+2}f^{2}(q|A|^{2}-\overline{\lambda}_{1}) \\
+ & \left(1+\dfrac{2(q+1)+1}{2\varepsilon}\right)\int_{M}|A|^{2q+2}|\nabla f|^{2}.
\end{aligned}
\end{eqnarray*}
From condition \eqref{1.8} and $|A|^{2}\leq\dfrac{\overline{\lambda}_{1}}{q}$,
we obtain
\[
\int_{M}|A|^{2q+4}f^{2}\leq \overline{C}\int_{M}|A|^{2q+2}|\nabla f|^{2}.
\]
The proof now proceeds exactly as in the final part of the proof of Theorem
\ref{th2}.
\end{proof}

\vspace{.1cm}

\begin{proof} [Proof of Theorem \ref{th5}.] 
We will apply the Young inequality to a certain term in the inequality \eqref{2.8}.
In fact, substituting the inequality 
\[
q\int_{M}\theta Hf^{2}|\phi|^{2q+3} \leq q\left(\dfrac{1}{2} \int_{M}f^{2}|\phi|^{2q+4}+\dfrac{1}{2} \int_{M}\theta^{2}H^{2}f^{2}|\phi|^{2q+2} \right)
\]
in \eqref{2.8}, we have
\begin{eqnarray} \label{3.8}
\begin{aligned}
&\left(1+\frac{2-n+n(q+2)}{nq}-\epsilon\right)
\int_{M}|\nabla|\phi|^{q}|^{2}f^{2}|\phi|^{2} \\
&\leq  q\int_{M}f^{2}|\phi|^{2q+4}+\dfrac{q}{2} 
\int_{M}f^{2}|\phi|^{2\alpha+4}+\dfrac{q}{2} 
\int_{M}\theta^{2}H^{2}f^{2}|\phi|^{2\alpha+2} \\
& -q\beta\int_{M}|\phi|^{2q+2}f^{2}+\dfrac{1}{\epsilon}
\int_{M}  |\phi|^{2q+2}|\nabla f|^{2}.
\end{aligned}
\end{eqnarray}
Multiplying the inequality \eqref{3.8} by
\[
\left(1+\dfrac{1}{q}\right)^{2}+\dfrac{\varepsilon(q+1)}{q^{2}},
\]
and multiplying the inequality \eqref{2.10} by
\[
1+\dfrac{2-n+n(q+2)}{nq}-\epsilon,
\]
and joining both, we get
\be
\no&& \left(1+\dfrac{2-n+n(q+2)}{nq}-\varepsilon\right)\int_{M}|\phi|^{2q+4}f^{2}\\
\no&&+\left(1+\dfrac{2-n+n(q+2)}{nq}-\varepsilon\right)(\overline{\lambda}_{1}+\beta)\int_{M}|\phi|^{2q+2}f^{2}\\
 & &\no\leq \left(\left(1+\dfrac{1}{q}\right)^{2}+\dfrac{\varepsilon(q+1)}{q^{2}}\right) q\int_{M}f^{2}|\phi|^{2q+4}\\
 \no&&+\left(\left(1+\dfrac{1}{q}\right)^{2}+\dfrac{\varepsilon(q+1)}{q^{2}}\right)\dfrac{q}{2} \int_{M}f^{2}|\phi|^{2\alpha+4}\\
 & &\no+\dfrac{q}{2}\left(\left(1+\dfrac{1}{q}\right)^{2}+\dfrac{\varepsilon(q+1)}{q^{2}}\right) \int_{M}\theta^{2}H^{2}f^{2}|\phi|^{2\alpha+2} \\
 \no &&-q\beta\left(\left(1+\dfrac{1}{q}\right)^{2}+\dfrac{\varepsilon(q+1)}{q^{2}}\right)\int_{M}|\phi|^{2q+2}f^{2}
 \no+\left(1+\dfrac{2}{\varepsilon}\right)\int_{M}|\phi|^{2q+2}|\nabla f|^{2}.
\en
That is,
\be
\no&& \left[\left(1+\dfrac{2-n+n(q+2)}{nq}-\varepsilon\right)-\left(\left(1+\dfrac{1}{q}\right)^{2}+\dfrac{\varepsilon(q+1)}{q^{2}}\right) \right.\\
\no&&\left.\times\left(q+\dfrac{q}{2}\right)\right]\int_{M}|\phi|^{2q+4}f^{2}
+\left[\left(1+\dfrac{2-n+n(q+2)}{nq}-\varepsilon\right)(\overline{\lambda}_{1}+\beta)\right.\\
&&\no+\left.\left(q\beta-\dfrac{q}{2}\right)\left(\left(1+\dfrac{1}{q}\right)^{2}
+\dfrac{\varepsilon(q+1)}{q^{2}}\right)\right]\int_{M}|\phi|^{2q+2}f^{2}\\
& &\no\leq \left(1+\dfrac{2}{\varepsilon}\right)\int_{M}|\phi|^{2q+2}|\nabla f|^{2}.
\en
Since $n=2$ or $n=3$, condition \eqref{1.9} in $q$ implies that
\[
\left(1+\dfrac{2-n+n(q+2)}{nq}\right)-\dfrac{3}{2}q\left(1+\dfrac{1}{q}\right)^{2}>0.
\]
Therefore, we can choose a constant $\epsilon_{1}>0$ so that, in an
$\epsilon$-neighborhood, we get
\[
\left(1+\frac{2-n+n(q+2)}{nq}-\varepsilon_{1}\right)-
\left(\left(1+\dfrac{1}{q}\right)^{2}+\dfrac{\varepsilon_{1}(q+1)}{q^{2}}\right)
\left(q+\dfrac{q}{2}\right)>0.
\]
On the other hand, from \eqref{1.10}, we have
\[
\overline{\lambda}_{1}>\dfrac{(q+1)^{2}n}{2(n+2nq+2)}(1-2n(H^{2}+c))-n(H^{2}+c).
\]
Thus, again we can choose a constant $\epsilon_{2}>0$ such that
\be
&&\no\left(1+\dfrac{2-n+n(q+2)}{nq}-\varepsilon_{2}\right)(\lambda_{1}+\beta)\\
&&\no+\left(q\beta-\dfrac{q}{2}\right)\left(\left(1+\dfrac{1}{q}\right)^{2}
+\dfrac{\varepsilon_{2}(q+1)}{q^{2}}\right)>0.
\en
Choosing $\epsilon=\min\{\epsilon_{1},\epsilon_{2}\}$, we obtain
\[
\int_{M}|\phi|^{2q+4}f^{2}+\int_{M}|\phi|^{2q+2}f^{2}\leq 
C\int_{M}|\phi|^{2q+2}|\nabla f|^{2},
\]
where $C$ is a positive constant that depends only on $n$, $H$, $q$ and
$\epsilon$. The rest of the proof follows as in the end of previous results.
\end{proof}

\bibliographystyle{amsplain}

\Addresses

\end{document}